\documentclass[a4paper,12pt]{article}
\usepackage{a4wide}
\usepackage{amsmath}
\usepackage{amssymb}
\usepackage{amsthm}
\usepackage{latexsym}
\usepackage{graphicx}
\usepackage[english]{babel}
\usepackage{makeidx}

\newtheorem{obs} [subsection]{Remark}
\newtheorem{exm} [subsection]{Example}

\newtheorem{prop}[subsection]{Proposition}
\newtheorem{conj}[subsection]{Conjecture}
\newtheorem{teor}[subsection]{Theorem}
\newtheorem{lema}[subsection]{Lemma}
\newtheorem{cor} [subsection]{Corollary}
\newcommand{\Zng}{$\mathbb Z^n$-graded $S$-module}

\def\sdepth{\operatorname{sdepth}}
\def\qdepth{\operatorname{qdepth}}
\def\depth{\operatorname{depth}}
\def\supp{\operatorname{supp}}
\def\deg{\operatorname{deg}}

\def\pd{\operatorname{pd}}
\def\reg{\operatorname{reg}}
\begin{document}
\selectlanguage{english}
\frenchspacing

\large
\begin{center}
\textbf{On the quasi-depth of squarefree monomial ideals and the sdepth of the monomial ideal of independent sets of a graph}

Mircea Cimpoea\c s
\end{center}
\normalsize

\begin{abstract}
If $J\subset I$ are two monomials ideals, we give a practical upper bound for the Stanley depth of $J/I$, which we call it the \emph{quasi-depth} of $J/I$. We compute the quasi-depth of several classes of square free monomial ideals.
Also, we study the Stanley depth of the monomial ideal associated to the independent sets of a graph.

\noindent \textbf{Keywords:} Stanley depth, monomial ideal, independent sets of a graph.

\noindent \textbf{2010 Mathematics Subject
Classification:} 13C15, 13P10, 13F20, 05C69.\end{abstract}
\section*{Introduction}

Let $K$ be a field and $S=K[x_1,\ldots,x_n]$ the polynomial ring over $K$.
Let $M$ be a \Zng. A \emph{Stanley decomposition} of $M$ is a direct sum $\mathcal D: M = \bigoplus_{i=1}^rm_i K[Z_i]$ as a $\mathbb Z^n$-graded $K$-vector space, where $m_i\in M$ is homogeneous with respect to $\mathbb Z^n$-grading, $Z_i\subset\{x_1,\ldots,x_n\}$ such that $m_i K[Z_i] = \{um_i:\; u\in K[Z_i] \}\subset M$ is a free $K[Z_i]$-submodule of $M$. We define $\sdepth(\mathcal D)=\min_{i=1,\ldots,r} |Z_i|$ and $\sdepth_S(M)=\max\{\sdepth(\mathcal D)|\;\mathcal D$ is a Stanley decomposition of $M\}$. The number $\sdepth_S(M)$ is called the \emph{Stanley depth} of $M$. In \cite{apel}, J.\ Apel restated a conjecture firstly given by Stanley in \cite{stan}, namely that $\sdepth_S(M)\geq\depth_S(M)$ for any \Zng $\;M$. This conjecture proves to be false, in general, for $M=S/I$ and $M=J/I$, where $I\subset J\subset S$ are monomial ideals, see \cite{duval}.

Herzog, Vladoiu and Zheng show in \cite{hvz} that $\sdepth_S(M)$ can be computed in a finite number of steps if $M=I/J$, where $J\subset I\subset S$ are monomial ideals.  However, it is difficult to compute this invariant, even in some very particular cases. In \cite{rin}, Rinaldo give a computer implementation for this algorithm, in the computer algebra system $CoCoA$ \cite{cocoa}. However, it is difficult to compute this invariant, even in some very particular cases.  For instance in \cite{par} Biro et al. proved that $\sdepth(m)= \left\lceil n/2 \right\rceil$ where $m=(x_1,\ldots,x_n)$.

In the first section, we give an upper bound for the $\sdepth(J/I)$, where $I\subset J\subset S$ are two squarefree monomial ideals, in numerical terms of the associated poset $\mathcal P_{J/I}$, see Proposition $1.1$ and Corollary $1.2$. We call this bound, the quasi-depth of $J/I$, and we denoted by $\qdepth(J/I)$. In Theorem $1.10$, we give a sharp upper bound for the $\sdepth(S/I)$, where $I\subset S$ is a monomial ideal generated by squarefree monomials of degree $m$. Our method can be useful in many cases. We consider the examples given by Duval, which disprove the Stanley Conjecture, see Example $1.4$ and $1.5$. 
In the second section, where we consider monomial ideals of independent sets associated to graphs. We give effective bounds for the sdepth of quotient ring of these ideals, see Theorem $2.4$ and Corollary $2.6$. We also propose an explicit formula, see Conjecture $2.7$.

\footnotetext[1]{The support from grant ID-PCE-2011-1023 of Romanian Ministry of Education, Research and Innovation is gratefully acknowledged.}

\newpage
\section{Main results.}

We recall a construction firstly presented by us in \cite{ciclu}.  Let $\mathcal P \subset 2^{[n]}$.  If $F,G\in\mathcal P$ with $F\subset G$, the interval $[F,G]$ is the set $\{H\subset \mathcal P\;:\; F\subset H\subset G\}$. Let $\mathbf P: \mathcal P=\bigcup_{i=1}^r [F_i,G_i]$ be a partition of $\mathcal P$, i.e. $[F_i,G_i]$ are pairwise disjoint intervals. We define the \emph{Stanley depth} of $\mathbf P$, the number $\sdepth(\mathbf P)=\min_{i=1}^r\{|G_i|\}$. We define the Stanley depth of $\mathcal P$, the number $\sdepth(\mathcal P)=\max\{\sdepth(\mathbf P)\;:\; \mathbf P $ is a partition of $\mathcal P\}$.

For each $0\leq k\leq n$, we denote $\mathcal P_k=\{A\in\mathcal P\;:\;|A|=k\}$ and $\alpha_k=|\mathcal P_k|$. Assume that $\sdepth(\mathcal P)=d$. It follows that $\mathcal P$ admits a partition $\mathcal P=\bigcup_{i=1}^r [F_i,G_i]$ with $|G_i|\geq d$ for all $i$. Moreover, by \cite[Proposition 2.2]{ciclu}, we can assume that there exists some $1\leq q\leq r$ such that $|G_j|=d$ for all $j\leq q$ and $|F_j|\geq d$ for all $q<j\leq r$. See also, \cite[Theorem 3.5]{rin}.

We define $\beta_0=\alpha_0$, $\beta_1=\alpha_1 - \beta_0 \binom{d}{1}$ and $\beta_k=\alpha_k - \beta_0 \binom{d}{k} - \beta_1 \binom{d-1}{k-1} - \cdots - \beta_{k-1}\binom{d-k}{1}$, for all $1\leq k\leq d$. By the proof of \cite[Theorem 2.4]{ciclu}, we get 
$\beta_k=|\{i\;:|F_i|=k\}$, for all $0\leq k\leq d$. Obviously, the $\beta_k$'s are non-negative. Thus, we have the following:

\begin{prop} (\cite[Theorem 2.4]{ciclu})
If $\mathcal P\subset 2^{[n]}$ with $\sdepth(\mathcal P)\geq d$, then $\beta_k\geq 0$, for all $0\leq k\leq d$.
\end{prop}

Note that the definition of $\beta_k$'s do not depend on the partition of $\mathcal P$. Let $\mathcal P\subset 2^{[n]}$ be an arbitrary poset and let $0\leq d\leq n$ be an integer. We say that the \emph{quasi-depth} of $\mathcal P$ is at least $d$, and we write $\qdepth(\mathcal P)\geq d$, if and only if, by definition, the numbers $\beta_0,\ldots,\beta_d$, defined as above, are non-negative. We say that $\qdepth(\mathcal P)=d$ if $d\leq n$ is the largest integer with the property that $\beta_0,\ldots,\beta_d$ are non-negative. As a direct consequence of Proposition $1.1$, we get:

\begin{cor}
$\sdepth(\mathcal P)\leq \qdepth(\mathcal P)$.
\end{cor}

Of course, the computation of the invariant $\qdepth(\mathcal P)$ is much easier than the computation of $\sdepth(\mathcal P)$, both from a theoretical point of view and, also, practical. Similar techniques were presented in \cite[Section $7.2.1$]{kat} and \cite[Section 5.2]{ichim}, with other terminology and in a more general context.

We recall the method of Herzog, Vladoiu and Zheng \cite{hvz} for computing the Stanley depth of $S/I$ and $I$, where $I$ is a squarefree monomial ideal. Let $G(I)=\{u_1,\ldots,u_s\}$ be the set of minimal monomial generators of $I$. We define the following two posets:
\[ \mathcal P_I:=\{C\subset [n]:\; \supp(u_i)\subset C \;\emph{for\;some}\;i\;\}\;\emph{and}\; 
\mathcal P_{S/I}:=2^{[n]}\setminus \mathcal P_I. \]
Also, if $I\subset J$ are two squarefree monomials ideals, we define $\mathcal P_{J/I}:=\mathcal P_{J}\cap \mathcal P_{S/I}$. Herzog Vladoiu and Zheng proved in \cite{hvz} that $\sdepth(J/I)=\sdepth(\mathcal P_{J/I})$. We define the \emph{quasi-depth} of $J/I$, the number $\qdepth(J/I):=\qdepth(\mathcal P_{J/I})$.

\begin{obs}
\emph{Let $I\subset J\subset S$ be two squarefree monomial ideals, and let $\mathcal P = \mathcal P_{J/I}$. 
As above, we denote $\mathcal P_k=\{A\in\mathcal P\;:\;|A|=k\}$ and $\alpha_k=|\mathcal P_k|$. Let $\bar J=(J,x_1^2,\ldots,x_n^2)$ and $\bar I=(I,x_1^2,\ldots,x_n^2)$. We claim that $\alpha_j=\dim_K (\bar J/\bar I)_j = H_{\bar J/\bar I}(j)$, for all $0\leq j\leq n$, where $H_{\bar J/\bar I}(-)$ is the \emph{Hilbert function} of $\bar J/\bar I$.}

\emph{Indeed, a monomial $u$ is in $\bar J \setminus \bar I$ , if and only if $u$ is square free and $u\in J\setminus I$. Also, the sets of $\mathcal P_k$ are in bijection with the square free monomials of degree $k$ from 
$J \setminus I$. Thus we proved the claim. In particular, if $J=S$, then $\alpha_j=H_{S/\bar I}(j)$. 
Also, if $I=(0)$, then $\alpha_j=H_{\bar J}(j)$.}
\end{obs}

\begin{exm}
\emph{We consider the ideal $I=(x_{13}x_{16}
 ,x_{12}x_{16}
,x_{11}x_{16}
,x_{10}x_{16}
,x_{9}x_{16}
,x_{8}x_{16}
,x_{6}x_{16}, \linebreak
x_{3}x_{16}
,x_{1}x_{16}
,x_{13}x_{15}
,x_{12}x_{15}
,x_{11}x_{15}
,x_{10}x_{15}
,x_{9}x_{15}
,x_{8}x_{15}
,x_{3}x_{15}
,x_{13}x_{14}
,x_{12}x_{14}
,x_{11}x_{14}
,x_{10}x_{14},\linebreak
x_{9}x_{14}
,x_{8}x_{14}
,x_{10}x_{13}
,x_{9}x_{13}
,x_{8}x_{13}
,x_{6}x_{13}
,x_{3}x_{13}
,x_{1}x_{13}
,x_{10}x_{12}
,x_{9}x_{12}
,x_{8}x_{12}
,x_{3}x_{12}
,x_{10}x_{11}
,x_{9}x_{11},\linebreak
x_{8}x_{11}
,x_{6}x_{10}
,x_{3}x_{10}
,x_{1}x_{10}
,x_{3}x_{9}
,x_{5}x_{7}
,x_{3}x_{7}
,x_{2}x_{7}
,x_{1}x_{7}
,x_{5}x_{6}
,x_{2}x_{6}
,x_{1}x_{6}
,x_{4}x_{5}
,x_{3}x_{5}
,x_{1}x_{4}, \linebreak
x_{4}x_{15}x_{16}
,x_{2}x_{15}x_{16}
,x_{2}x_{4}x_{15}
,x_{6}x_{7}x_{14}
,x_{1}x_{5}x_{14}
,x_{4}x_{12}x_{13}
,x_{2}x_{12}x_{13}
,x_{2}x_{4}x_{12}
,x_{6}x_{7}x_{11}
,x_{1}x_{5}x_{11}, \linebreak
x_{4}x_{9}x_{10}
,x_{2}x_{9}x_{10}
,x_{2}x_{4}x_{9}
,x_{6}x_{7}x_{8}
,x_{1}x_{5}x_{8}) \subset S=K[x_1,\ldots,x_{16}]$.}

\emph{According to \cite[Theorem 3.5]{duval}, it follows that $\sdepth(S/I)<\depth(S/I)$. The explicit computation of $\sdepth(S/I)$ is too hard for any computer. However, if we consider the poset $\mathcal P = \mathcal P_{S/I}$, one can easily check that $\alpha_0=1$, $\alpha_1=15$, $\alpha_2=71$, $\alpha_3=98$ and $\alpha_4=42$. For $d=4$, we have $\beta_0=1$, $\beta_1=\alpha_1-4\beta_0 = 12$, $\beta_2=\alpha_2-\binom{4}{2}\beta_0-\binom{3}{1}\beta_1 = 72 - 6 - 36 =30$ and $\beta_3=\alpha_3 - 4\beta_0 - 3\beta_1 - 2\beta_2 = 98-4-36-60 = -2<0$. Therefore, $\qdepth(S/I)<4 = \depth(S/I)=\dim(S/I)$. Now, let $\mathcal P = \mathcal P_I$. It follows that $\alpha_0=\alpha_1=0$, $\alpha_2=\binom{16}{2}-71 = 49$, $\alpha_3=\binom{16}{3}-98 = 462$, $\alpha_4=\binom{16}{4}-42 = 1778$ and $\alpha_k=\binom{16}{k}$ for all $5\leq k\leq 16$. One can easily check that $\qdepth(I)\geq 5$ so, we can expect that $\sdepth(I)\geq \depth(I)=5$. Note that, the Stalney conjecture is still open in the case of ideals.}
\end{exm}

\begin{exm}
\emph{Let $S=K[x_1,\ldots,x_6]$, $I=(x_1x_4x_5, x_4 x_6, x_2x_3x_6)\subset S$ and 
$J=(x_1x_2, x_1x_5, \linebreak x_1x_6, x_2x_3,x_2x_4, x_4x_6)\subset S$. According to \cite[Remark 3.6]{duval}, we have $\sdepth(J/I)=3<\depth(J/I)=4$. However, if we denote $\mathcal P=\mathcal P_{J/I}$, we have $\alpha_0=\alpha_1=0$, $\alpha_2=5$, $\alpha_3=10$ and $\alpha_4=5$. By straightforward computations, we get $\qdepth(J/I)=4$.}
\end{exm}

\begin{prop}
If $\mathbf m=(x_1,\ldots,x_n)$, then $\qdepth(\mathbf m)= \left\lceil \frac{n}{2} \right\rceil$.
\end{prop}

\begin{proof}
Since $\sdepth(\mathbf m)= \left\lceil \frac{n}{2} \right\rceil$, see \cite[Theorem 2.2]{par}, it is enough to prove that for $d:=\left\lceil \frac{n}{2} \right\rceil+1$, the numbers $\beta_k$'s are not all non-negative. Let $\mathcal P=\mathcal P_{\mathbf m}$. We have $\alpha_0=0$ and $\alpha_k=\binom{n}{k}$ for all $1\leq k\leq n$.
We get, $\beta_0=0$, $\beta_1=\alpha_1=n$. 
Therefore $\beta_2=\alpha_2 - \beta_1\binom{d-1}{1} = \binom{n}{2} - n(d-1) = \frac{n}{2}(n-1-2\left\lceil \frac{n}{2} \right\rceil)<0$ and thus we are done.
\end{proof}

\begin{prop}(See \cite[Theorem 1.1]{cim})
Let $I_{n,m}$ be the monomial ideal generated by all the square free monomials of degree $m$.
Then $\sdepth(I_{n,m})\leq \qdepth(I_{n,m}) \leq \left\lceil \frac{n-m}{m+1} \right\rceil + m - 1$.
\end{prop}

\begin{proof}
Let $0\leq d\leq n$ be an integer. Note that $\alpha_0=\alpha_1=\cdots=\alpha_{m-1}=0$ and $\alpha_k=\binom{n}{k}$ for all $m\leq k\leq n$. It follows that $\beta_0=\beta_1=\cdots=\beta_{m-1}=1$ and $\beta_m=\alpha_m=\binom{n}{m}$. Therefore $\beta_{m+1}=\alpha_{m+1}- d \beta_m = \binom{n}{m+1} - d \binom{n}{m}$. One can easily check that if $d\geq \left\lceil \frac{n-m}{m+1} \right\rceil + m$, then $\beta_{m+1}<0$. Therefore, $\qdepth(I_{n,m})\leq \left\lceil \frac{n-m}{m+1} \right\rceil + m - 1$.
\end{proof}

\begin{lema}
Let $n\geq 2$, $1\leq d\leq n$ be two integers. Let $\alpha_k=\binom{n}{k}$, for all $0\leq k\leq d$. Let $\beta_0=\alpha_0$ and $\beta_k=\alpha_k - \beta_0 \binom{d}{k} - \beta_1 \binom{d-1}{k-1} - \cdots - \beta_{k-1}\binom{d-k}{1}$, for all $1\leq k\leq d$. Then $\beta_k=\binom{n+k-d-1}{k}$, for all $k\leq n$.
\end{lema}

\begin{proof}
In order to prove the Lemma, one has to check the identity: 
\[ \binom{n}{k} = \binom{n-d}{0}\binom{d}{k} + \binom{n-d+1}{1}\binom{d-1}{k-1}+\cdots + \binom{n+k-d-1}{k}\binom{d-k+1}{0},\]
for all $1\leq k\leq d$, which can be done be identify the coefficients of $t^k$ in the identity $(1+t)^n = (1+t)^{n-d}(1+t)^d$.
\end{proof}

\begin{lema}
Let $\mathcal P = \mathcal P'\cup \mathcal P'' \subset 2^{[n]}$ such that $\mathcal P'\cap \mathcal P''=\emptyset$. Then $\sdepth(\mathcal P)\geq \min\{\sdepth(\mathcal P'),\sdepth(\mathcal P'')\}$.
\end{lema}

\begin{proof}
It is enough to notice that, given two partitions of $\mathcal P'$ and $\mathcal P''$, one can obtain a partition of $\mathcal P$ as the union of them.
\end{proof}

\begin{teor}
Let $I\subset S$ be a squarefree monomial ideal generated in degree $m<n$ with $g=|G(I)|$.

$(a)$ If $\binom{n+m-d-1}{m}<g$, then $\qdepth(S/I)\leq d-1$. In particular, $\sdepth(S/I)\leq d-1$.

$(b)$ In particular, if $\binom{n-1}{m}<g$, then $\qdepth(S/I)= \sdepth(S/I)=m-1$. 
\end{teor}

\begin{proof}
$(a)$ Let $\mathcal P=\mathcal P_{S/I}$. With the notations used in the first part of the section, we have $\alpha_k=\binom{n}{k}$, for all $0\leq k\leq m-1$, and $\alpha_m = \binom{n}{m}-g$. By Lemma $1.8$, we get $\beta_k=\binom{n+k-d-1}{k}$ for all $0\leq k<m$ and $\beta_m = \binom{n+m-d-1}{m}-g$. Therefore, is $\beta_m<0$, it follows that $\qdepth(S/I)\leq d-1$ and thus the required conclusion follows.

$(b)$ By $(a)$, we get $\sdepth(S/I)\leq \qdepth(S/I)\leq m-1$. In order to prove the converse, it is enough to notice that we can find partitions $\mathbf P_k$ of the poset $2^{[n]}$ with $\sdepth(\mathbf P_k)=k$ for any $0\leq k\leq n$. We can write $\mathcal P:=\mathcal P_{S/I} =\mathcal P' \cup \mathcal P''$, where $\mathcal P'=\{F\subset [n]\;:\; |F|<m \}$ and $\mathcal P''=\{F\in\mathcal P\;:\;|P|\geq m\}$. From the previous remark, $\sdepth(\mathcal P')=m-1$. On the other hand, $\sdepth(\mathcal P'')\geq m$. It follows, by Lemma $1.9$, that $\sdepth(\mathcal P)\geq m-1$, and thus $\sdepth(S/I)\geq m-1$.
\end{proof}

In \cite{pop}, D.\ Popescu proved a similar result for $\sdepth(I)$, namely, if  $\binom{n-1}{m}<g$ then $\sdepth(I)=\depth(I)=m$, see \cite[Remark 2.7]{pop}.

\section{Monomial ideal of independents sets of a graph}

Let $n\geq 3$ be an integer and let $G=(V,E)$ be a graph with the vertex set $V=[n]$ and edge set $E$. A set of vertices $S$ is \emph{independent} if there are no elements $i$ and $j$ of $S$ such that $\{i,j\}\in E$. We denote by $Ind(G)$ the set of all the independent sets of $G$. Let $\alpha(G)$ be the maximal cardinality of an independent set of $G$, called the \emph{independence number} of the graph $G$. Let $T:=K[s_i,t_i:\; i\in [n]]$ be the ring of polynomials in two sets of $n$ variables. For any $S\in Ind(G)$, we consider the monomial $m_S:=\prod_{i\in S}s_i \prod_{i\notin S}t_i$. Let $I:=(m_S\;:\;S\in Ind(G))\subset T$. The ideal $I$ is called the monomial ideal of independent sets associated to the graph $G$. The algebraic invariants of $I$ were studied by Olteanu in \cite{oana}. Later on, Cook II studied some generalization of this ideals, see \cite{cook}. We recall several results regarding the monomial ideal of independent sets. 


For any monomial $m\in G(I)$ with $m=m_S$, where $S\in Ind(G)$, we denote $m^{(s)}=\prod_{j\in S}s_j$ and $m^{(t)}=\prod_{j\notin S}t_j$ the $s$-part, respectively the $t$-part of $m$. Moreover, we denote $\deg_s(m)=\deg(m^{(s)})$ and $\deg_t(m)=\deg(m^{(t)})$. We consider the lexicographic order on $K[s_1,\ldots,s_n]$ induced by $s_1>s_2>\cdots>s_n$. Next, we define the monomial order $\succ$ on the monomials of $T$, given by: $m\succ m'$ if and only if $deg_s(m)<deg_s(m')$ or, $deg_s(m)=deg_s(m')$ and $m^{(s)}>_{lex}m'^{(s)}$. We assume that 
$G(I)=\{m_1,\ldots,m_r\}$, where $m_1\succ m_2\succ \cdots \succ m_r$. Assume $m_i=m_{S_i}$ for $S_i\in Ind(G)$.
Let $I_i:=(m_1,\ldots,m_i)$ for all $i\in [r]$. Let $a_k$ be the number of the number of independent
sets with $k$ elements of the graph $G$.
With this notations, we have the following result.

\begin{teor}(\cite[Theorem 2.2]{oana},\cite[Corollary 2.3]{oana})
$(I_{i-1}:m_i)=(t_r\;r\in S_i)$, for all $i>1$. In particular, $I$ has linear quotients and thus linear resolution. Moreover:

(a) $\reg(I)=n$.

(b) $\beta_i(I) = \sum_{k=0}^{\alpha(G)}a_k\binom{k}{i}$.

(c) $\pd(T/I)=\alpha(G)+1$.

(d) $\dim(T/I)=2n-2$.

(e) $\depth(T/I)=2n-\alpha(G)-1$.

(f) $T/I$ is Cohen-Macaulay if and only if $G$ is the complete graph.
\end{teor}

In \cite{asia}, Asia Rauf proved the following result:

\begin{lema}
Let $0 \rightarrow U \rightarrow M \rightarrow N \rightarrow 0$ be a short exact sequence of $\mathbb Z^n$-graded $S$-modules. Then
$ \sdepth(M) \geq \min\{\sdepth(U),\sdepth(N) \}$ .
\end{lema}

\begin{obs}
\emph{Let $\mathcal P$ be a set of squarefree monomials from $T$. If $u,v\in T$ with $u|v$ are two monomials, we denote $[u,v]:=\{w\;$ monomial $\;:\; u|w$ and $w|v\}$. Let $\mathbf P:\mathcal P=\bigcup_{i=1}^r [u_i,v_i]$ be a partition of $\mathbf P$. We denote $\sdepth(\mathbf P):=\min_{i\in [r]} \deg(v_i)$. Also, we define the Stanley depth of $\mathcal P$, to be the number $\sdepth(\mathcal P) = \max\{\sdepth(\mathbf P):\; \mathbf P \; \emph{is\; a\; partition\; of} \; \mathcal P \}.$}

\emph{We define the following two posets:
\[ \mathcal P_I:=\{m \in T\; square \; free\; :\; u_i|m \;\emph{for\;some}\;i\;\}\;\emph{and}\;\] 
\[ \mathcal P_{T/I}=\{m \in T\; square \; free\; :\; u_i\nmid m \;\emph{for\;all}\;i\;\} \]}
\emph{Herzog Vladoiu and Zheng proved in \cite{hvz} that $\sdepth(I)=\sdepth(\mathcal P_{I})$ and $\sdepth(T/I)=\sdepth(\mathcal P_{T/I})$. Let $\mathcal P=\mathcal P_{T/I}$.
Now, for $d\in\mathbb N$ and $u \in \mathcal P$, we denote
\[ \mathcal P_d = \{u \in\mathcal P\;:\; \deg(u)=d \}\;,\;\mathcal P_{d,u} = \{ u\in\mathcal P_d\;:\; u|v \}. \]}
\emph{Note that if $u \in \mathcal P$ such that $P_{d,u}=\emptyset$, then $\sdepth(\mathcal P)<d$. Indeed, let $\mathbf P:\mathcal P=\bigcup_{i=1}^r [u_i,v_i]$ be a partition of $\mathcal P$ with $\sdepth(\mathcal P)=\sdepth(\mathbf P)$. Since $\sigma\in\mathcal P$, it follows that  $u \in [u_i,v_i]$ for some $i$. If $\deg(v_i)\geq d$, then it follows that $\mathcal P_{u,d}\neq\emptyset$, since there are monomials in the interval $[u_i,v_i]$ with degree $d$ which are divisible by $u$, a contradiction. Thus, $\deg(v_i)<d$ and therefore $\sdepth(\mathcal P)<d$.}

\emph{Take $u=s_1\cdots s_{n-1}t_1\cdots t_{n-1}\notin I$. Since $s_nu \in I$ and $t_nu\in I$, it follows that $\mathcal P_{u,2n-1}=\emptyset$ and therefore $\sdepth(T/I)\leq 2n-2=\dim(T/I)$. Of course, this inequality could be also deduced, more directly, from the fact that $I$ is not a principal ideal.}
\end{obs}

The main result of this section is the following Theorem.

\begin{teor}
$\sdepth(T/I)\geq \depth(T/I)$.
\end{teor}

\begin{proof}
We use induction on $n\geq 1$. If $n=1$, then $I=(s_1,t_1)$ and thus $\sdepth(T/I) = \depth(T/I) =2$. Assume $n\geq 2$. 
We consider the short exact sequence:
\[ 0 \rightarrow T/(I:t_n) \rightarrow T/I \rightarrow T/(I,t_n) \rightarrow 0. \]

Let $T'=K[t_1,\ldots,t_{n-1},s_1,\ldots,s_{n-1}]$ and let $J\subset T'$ be the ideal of the independent sets of the graph $G'=G\setminus\{n\}$. 
Note that $\alpha(G')\leq\alpha(G)$. Also $(I:t_n) = J T$ and therefore, according to \cite[Lemma 3.6]{hvz}, $\sdepth(T/(I:t_n)) = \sdepth_{T'}(T'/J) \geq 2 + \depth_{T'}(T'/J) = 2 + 2(n-1) - \alpha(G') - 1 = 2n - \alpha(G') - 1 \geq 2n-\alpha(G)-1$. 

On the other hand, $(I,t_n)$ is generated by $t_n$ and all the monomials of the form $m_S$ with  $S\in Ind(G)$ and $n\in S$. Let $G''$ be the graph obtained from $G$ be deleting the vertex $\{n\}$ and all the vertices adiacent to $\{n\}$. For convenience, we may assume that $G''$ is a graph on the vertex set $[m]$ with $m<n$. Note that $Ind(G'')=\{S\setminus \{n\} :\;S\in Ind(G),\; n\in S\}$. Thus, $\alpha(G'')\leq \alpha(G)$ and, moreover, $(I,t_n)= (t_n,s_nL)$, where $L$ is the ideal of independent sets of $G''$. Using the induction hypothesis, \cite[Lemma 3.6]{hvz} and \cite[Theorem 1.4]{mir}, we get $\sdepth(T/(I,t_n))=1+2m-1-\alpha(G'') + 2(n-m-1) = 2n-2 - \alpha(G'')\geq 2n-1-\alpha(G)$. Therefore, by Lemma $2.2$ we are done.
\end{proof}

\begin{cor}
If $G$ is the complete graph, then $\sdepth(T/I)=\depth(T/I)=2n-2$.
\end{cor}

\begin{proof}
Since $G$ is the complete graph, we $\alpha(G)=1$ and, moreover: 
$$I=(t_1\cdots t_n, s_1t_2\cdots t_n, \ldots, s_nt_1 \cdots t_{n-1}).$$ 
According to Theorem $2.1$ and Theorem $2.4$, it follows that $\sdepth(T/I)\geq \depth(T/I)=2n-2$.
On the other hand, by Remark $2.3$, we have $\sdepth(T/I)\leq 2n-2$, and thus we are done.
\end{proof}


Let $g:=|G(I)|$ the number of minimal monomial generators of $I$. Note that $|G(I)|\leq \sum_{i=0}^{\alpha(G)}\binom{n}{i}$. We denote $\gamma(G)=\max\{d:\; \binom{3n-d-1}{n}\geq g\}$. Note that $n-1 \leq \gamma(G) \leq 2n-2$. Indeed, for $d=n-1$, we have $|G(I)|\leq 2^n$ and $\binom{3n-d-1}{n}=\binom{2n}{n}\geq 2^n$ for all $n\geq 1$. On the other hand, if $d>2n-2$, then $3n-d-1<n$ and therefore $\binom{3n-d-1}{n}=0<1\leq |G(I)|$. 

As a direct consequence of Theorem $1.10$ and Theorem $2.4$ we get the following corollary.

\begin{cor}
$\depth(T/I)\leq \sdepth(T/I)\leq \gamma(G) \leq 2n-2$.
\end{cor}

Our computer experimentations (see \cite{cocoa} and \cite{rin}) lead us to the following conjecture:

\begin{conj}
$\sdepth(T/I)=\gamma(G)$.
\end{conj}

\begin{obs}
\emph{Note that if $I$ is minimally generated by $g>\binom{2n-1}{n}$ squarefree monomials of degree $n$, Theorem $1.10(b)$ implies $\sdepth(T/I)=n-1$. So, the conjecture holds in this very particular case. This is the case, for instance, when $G$ is the discrete graph on $n\leq 3$ vertices. Also, by Corollary $2.5$, the conjecture holds for $G$ the complete graph.}

\emph{Moreover, if $\gamma(G)=2n-\alpha(G)-1$, then, by Corollary $2.6$, it follows that $\sdepth(T/I)=\depth(T/I)$. Note that $\gamma(G)=2n-\alpha(G)-1$ is equivalent to $\binom{n+\alpha(G)-1}{n} < g$. This is a special condition for a graph $G$, which holds rarely.}
\end{obs}

\begin{exm}
\emph{Let $G=(V,E)$ be the cycle of length $4$, i.e. $V=[4]$ and $E=\{\{1,2\},\{2,3\},\linebreak \{3,4\},\{4,1\}\}$. The independent sets of $G$ are $\emptyset, \{1\},\{2\},\{3\},\{4\}, \{1,3\}$ and $\{2,4\}$. The associated ideal $I\subset T=K[s_1,s_2,s_3,s_4,t_1,t_2,t_3,t_4]$ of independent sets is generated by $m_1=t_1t_2t_3t_4$, $m_2=s_1t_2t_3t_4$, $m_3=s_2t_1t_3t_4$, $m_4=s_3t_1t_2t_4$, $m_5=s_4t_1t_2t_3$, $m_6=s_1s_3t_2t_4$ and $m_7=s_2s_4t_1t_3$. By Theorem $2.1$, one has $\reg(I)=4$, $\pd(T/I)=3$, $\dim(T/I)=6$ and $\depth(T/I)=5$. Moreover, $g=|G(I)|=7$ and thus $\gamma(G)=5$, since $\binom{6}{4}=15 >7$, but $\binom{5}{4}=5<7$. According to Corollary $2.6$, we get $\sdepth(T/I)=5$.}

\emph{Let $G=(V,E)$ by the line ideal of length $5$, i.e. $V=[5]$ and $E=\{\{1,2\},\{2,3\}, \{3,4\},\linebreak \{4,5\}\}$. Note that $Ind(G)=\{ \emptyset, \{1\}, \{2\}, \{3\}, \{4\}, \{5\}, \{1,3\}, \{1,4\}, \{1,5\}, \{2,4\}, \{2,5\},\linebreak \{3,5\}, \{1,3,5\}\}$. We have $\alpha(G)=3$ and $g=|Ind(G)|=13$. Note that $\gamma(G)=7>\depth(T/I)=6$, since $\binom{7}{5}=12>13$ and $\binom{6}{5}=6<13$.}
\end{exm}

\vspace{2mm} \noindent {\footnotesize
\begin{minipage}[b]{15cm}
Mircea Cimpoea\c s, Simion Stoilow Institute of Mathematics, Research unit 5, P.O.Box 1-764,\\
Bucharest 014700, Romania, E-mail: mircea.cimpoeas@imar.ro
\end{minipage}}


\begin{thebibliography}{99}
\bibitem[1]{apel} J.\ Apel, \textit{On a conjecture of R. P. Stanley; Part II - Quotients Modulo Monomial Ideals}, 
      J. of Alg. Comb. 17, (2003), 57--74.
\bibitem[2]{par} C.\ Biro, D.\ M.\ Howard, M.\ T.\ Keller, W.\ T.\ Trotter, S.\ J.\ Young,
\textit{Interval partitions and Stanley depth}, Journal of Combinatorial Theory, \textbf{Series A 117 (4)}, (2010), 475--482.
\bibitem[3]{ciclu} M.\ Cimpoeas, \textit{On the Stanley depth of edge ideals of line and cyclic graphs}, 
Romanian Journal of Math. and Computer Science \textbf{5(1)}, (2015), 70--75.              
\bibitem[4]{mir} M.\ Cimpoeas, \textit{Stanley depth of monomial ideals with small number of generators},
       Central European Journal of Mathematics, \textbf{vol. 7, no. 4}, (2009), 629--634.
\bibitem[5]{mirci} M.\ Cimpoeas, \textit{Several inequalities regarding Stanley depth}, Romanian Journal of Math. and Computer Science \textbf{2(1)}, (2012), 28--40.              

\bibitem[6]{cim} M.\ Cimpoeas, \textit{Stanley depth of square free Veronese ideals}, An. St. Univ. Ovidius, \textbf{Vol. 21(3)}, (2013), 67--71. 

\bibitem[7]{cook} D.\ Cook II, \emph{The uniform face ideals of a simplicial complex}, 
http://arxiv.org/pdf/1308.1299.pdf, Preprint (2013). 

\bibitem[8]{cocoa} CoCoATeam, \textit{CoCoA: a system for doing Computations in Commutative Algebra}.

\bibitem[9]{duval} A.\ M.\ Duval, B.\ Goeckneker, C.\ J.\ Klivans, J.\ L.\ Martine, \emph{A non-partitionable Cohen-Macaulay simplicial complex}, http://arxiv.org/pdf/1504.04279, Preprint (2015).

\bibitem[10]{hvz} J.\ Herzog, M.\ Vladoiu, X.\ Zheng, \textit{How to compute the Stanley depth of a monomial ideal},
           Journal of Algebra \textbf{322(9)}, (2009), 3151--3169.

\bibitem[11]{ichim} B.\ Ichim, , L.\ Katth\"an, J.\ J.\ Moyano-Fernandez, \emph{How to compute the Hilbert depth of a module}, Mathematical Programming, in press, (2015).

\bibitem[12]{kat} L.\ Katth\"an, \emph{Stanley depth and simplicial spanning trees}, Journal of Algebraic Combinatorics, Journal of Algebraic Combinatorics 42(2), (2015), 507--535.

\bibitem[13]{oana} O.\ Olteanu, The monomial ideal of independents sets associated to a graph, http://arxiv.org/pdf/1307.3050.pdf, Preprint 
(2013).




\bibitem[14]{pop} D.\ Popescu, \emph{Depth of factors of square free monomial ideals}, Proc. Amer. Math. Soc. \textbf{142} (2014), 1965--1972.
\bibitem[15]{asia} A.\ Rauf, \textit{Depth and sdepth of multigraded module}, Communications in Algebra,
\textbf{vol. 38, Issue 2}, (2010), 773--784.
\bibitem[16]{stan} R.\ P.\ Stanley, Linear Diophantine equations and local cohomology, Invent. Math. \textbf{68}, (1982), 175--193.    
\bibitem[17]{rin} G.\ Rinaldo, \textit{An algorithm to compute the Stanley depth of monomial ideals}, Le Matematiche, \textbf{Vol. LXIII (ii)}, (2008), 243--256.

\end{thebibliography}
\end{document}